\documentclass{amsart}
\usepackage{amsfonts,amssymb,amsmath,amsthm}
\usepackage{url}
\usepackage{enumerate}

\newtheorem{thrm}{Theorem}[section]
\newtheorem{lem}[thrm]{Lemma}

\newtheorem{cor}[thrm]{Corollary}
\theoremstyle{definition}

\newtheorem{remark}[thrm]{Remark}
\newtheorem{example}[thrm]{Example}
\numberwithin{equation}{thrm}

\author{Tiberiu Dumitrescu }
\address{Facultatea de Matematica si Informatica,
University of Bucharest,
14 Academiei Str., Bucharest, RO 010014,
Romania
}
\email{tiberiu@fmi.unibuc.ro, tiberiu\_dumitrescu2003@yahoo.com}
\author{Cristodor Ionescu}
\address{Institute of Mathematics {Simion Stoilow} of the Romanian Academy,   P.O. Box 1-764, Bucharest, RO 014700, 
Romania
}
\email{Cristodor.Ionescu@imar.ro}

\begin{document}

\title[Running Head]{Some examples of two-dimensional regular rings}

\begin{abstract}
Let $B$ be a ring and $A=B[X,Y]/(aX^2+bXY+cY^2-1)$ where $a,b,c\in B$. 
We study the smoothness of $A$ over $B$,  and
the regularity of $B$ when $B$ is a ring of algebraic integers. 
\end{abstract}
\maketitle

\section{Introduction} \label{sect1}
In  \cite{R}, Roberts investigated the smoothness (over $\mathbb{Z}$) and the regularity of the ring
$\mathbb{Z}[X,Y]/(aX^2+bXY+cY^2-1), a,b,c\in \mathbb{Z}.$ 
He showed that smoothness depends on $a,b,c$ mod $2$ (cf. \cite[Theorem 1]{R}), while regularity 
depends on $a,b,c$ 
mod $4$ (cf. \cite[Theorem 2]{R}).

In this note, we use ideas from  \cite{R} to study the regularity of the
 ring $A:=B[X,Y]/(aX^2+bXY+cY^2-1)$, $a,b,c\in B$, where $B$ is a ring of algebraic integers. 
As expected, this regularity depends on $a,b,c$ mod $(\sqrt{2B})^2$ 
(see Corollary \ref{8}). Our main result (Theorem \ref{3})  gives a description of the 
singular locus of $B$.
On the way, we show that the smoothness of $A$ over $B$ can be easily described: if $B$ is an 
arbitrary ring, then $A$ is smooth over $B$ iff $a,c\in  \sqrt{(2, b)B}$ (Theorem \ref{1}).
Finally, Example \ref{14} suggests that our arguments  can be also used in certain  higher 
degree cases. 
Throughout this paper, all rings are commutative and unitary. For any undefined terminology our 
standard reference is \cite{M}.

\section{Smoothness} \label{ns}

Let $B$ be an arbitrary ring. The smoothness of $B[X,Y]/(aX^2+bXY+cY^2-1)$ over $B$ can be described easily.

\begin{thrm}\label{1}
Let $B$ be a ring and $a,b,c\in B$. Then $A=B[X,Y]/(aX^2+bXY+cY^2-1)$ is smooth over $B$ iff $a,c\in  \sqrt{(2, b)B}$.
\end{thrm}
\begin{proof}

Let $J'$ be the jacobian ideal of $A$ and set $f=aX^2+bXY+cY^2$.
By Euler's formula for homogeneous functions, we have 
$2f=X(\partial f/\partial X)+Y(\partial f/\partial Y)$,
hence $2\in J'$, because the image of $f$ in $A$ is $1$.
Moding out by $2B$, we may assume that $B$ has characteristic $2$.
We have to show that $A$ is smooth over $B$ iff      $a,c$ are nilpotent modulo $b$.
Set $C=B[X,Y]$.
Note that  $J'=JA$ where 
$$J=(bX,bY,f-1)C=(bX,bY,aX^2+cY^2-1)C =(b,aX^2+cY^2-1)C$$
because $b=a(bX)^2+c(bY)^2-b(aX^2+cY^2-1)$.
So $A$ is  smooth over $B$ iff $J=C$, cf. \cite[Proposition 5.1.9]{MR}.
Now $J=C$ iff $aX^2+cY^2-1$ is invertible modulo $b$ iff $a,c$ are nilpotent modulo $b$.
\end{proof}

\begin{cor}{\em (\cite[Theorem 1]{R}.)}
Let $A=\mathbb{Z}[X,Y]/(aX^2+bXY+cY^2-1)$. Then $A$ is smooth over $\mathbb{Z}$ iff $b$ is odd or $a,b,c$ are all even.
\end{cor}
\begin{proof}
By Theorem \ref{1}, $A$ is smooth over $\mathbb{Z}$ iff $a,c\in  \sqrt{(2, b)\mathbb{Z}}$ iff $b$ is odd or $a,b,c$ are all even.
\end{proof}

Similar results can be stated for any  ring of algebraic integers; here are two examples.

\begin{cor}\label{11}
Let $A=\mathbb{Z}[(1+\sqrt{-7})/2][X,Y]/(aX^2+bXY+cY^2-1)$
and set $\theta=(1+\sqrt{-7})/2$.
Then $A$ is smooth over $\mathbb{Z}[(1+\sqrt{-7})/2]$ iff 
one of the following cases occurs:

$(i)$ $b$ is not divisible by  $\theta$ or $\bar{\theta}$,

$(ii)$ $b$ is not divisible by  $\theta$ and  $a,b,c$ are all divisible by $\bar{\theta}$,

$(iii)$ $b$ is not divisible by  $\bar{\theta}$ and  $a,b,c$ are all divisible by ${\theta}$,

$(iv)$   $a,b,c$ are all divisible by $2$.
\end{cor}
\begin{proof}
We have $2={\theta}\bar{\theta}$. 
By Theorem \ref{1}, $A$ is smooth over $B=\mathbb{Z}[(1+\sqrt{-7})/2]$ iff $a,c\in  \sqrt{(2, b)B}=(\theta,b)B\cap (\bar{\theta},b)B$. 
As $\theta B$ and $\bar{\theta} B$ are maximal ideals,
the assertion is clear.
\end{proof}

\begin{cor}
Let $A=\mathbb{Z}[\theta][X,Y]/(aX^2+bXY+cY^2-1)$,  $\theta=(\sqrt{2}+\sqrt{6})/2$.
Then $A$ is smooth over $\mathbb{Z}[\theta]$ iff $b$ is not divisible by $1+\theta$ or $a,b,c$ are all divisible by $1+\theta$.
%
\end{cor}
\begin{proof}
It can be checked by PARI-GP (see \cite{P}) that $B=\mathbb{Z}[\theta]$ is a PID and $2B=(1+\theta)^4B$.  So $\sqrt{(2, b)B}=(1+\theta, b)B$.
Apply Theorem \ref{1}.
%
\end{proof}

\section{Regularity}

Throughout this section we fix the following notations.
Let $B$ be a ring of algebraic integers, that is, the integral closure of $\mathbb{Z}$ in a finite field extension of $\mathbb{Q}$. It is well-known (e.g. \cite[Chapter 6]{AW}) that $B$ is a Dedekind domain, hence $II^{-1}=B$ for every nonzero ideal $I$ of $B$.
Fix  $a,b,c\in B$. We study the regularity  of the ring  $$A=B[X,Y]/(aX^2+bXY+cY^2-1).$$
Set $g=aX^2+bXY+cY^2-1$ and $C=B[X,Y]$.

\begin{lem}\label{6}
If $Q\in Spec(A)$ and $Q\cap B\not\supseteq (2,b)B$, then $A_Q$ is a regular ring.
\end{lem}
\begin{proof}
Since $P:=Q\cap B\not\supseteq (2,b)B$,  Theorem \ref{1} shows that the composed morphism
$$B_{P}\rightarrow  B_{P}\otimes_B A =B_P[X,Y]/(g)\rightarrow A_{Q}$$ is smooth. Hence $A_{Q}$ is regular, because $B$ is regular.
\end{proof}

Let $\Gamma$ be   the (finite) set of prime ideals $P$ of $B$ such that $P\supseteq (2,b)B$ and $P\not\supseteq (a,c)B$. By Theorem \ref{1}, $A$ is smooth over $B$ iff $\Gamma=\emptyset.$

\begin{lem}\label{5}
If $Q\in Spec(A)$, then $Q\cap B\in \Gamma$ iff $Q\cap B\supseteq (2,b)B$.
\end{lem}
\begin{proof}
Assume that $P:=Q\cap B\supseteq (2,b)B$. As $(a,b,c)A=A$, it follows that $Q\not\supseteq (a,c)B$, so $P\not\supseteq (a,c)B$, that is, $P\in \Gamma$. The converse is obvious.
\end{proof}

Let $P\in \Gamma$.
Since $B$ is a ring of algebraic integers and $2\in P$, it follows that $B/P$ is a finite (thus perfect)  field of  characteristic $2$. 
If $z\in B$, let $\bar{z}$ denote its image in $B/P$. Let
$d,e\in B$ such that 
\begin{equation}\label{eq5}
\bar{d}^2=\bar{a} \mbox{ and } \bar{e}^2=\bar{c}. 
\end{equation}
Note that $d,e\in B$ are uniquely determined modulo $P$.
If $a\not\in P$ (hence $d\not\in P$), denote 
\begin{equation}\label{eq1}
F_P:=d^2g(\frac{-eY-1}{d},Y)=a(eY+1)^2+bd(-eY-1)Y+d^2cY^2-d^2=
\end{equation}
$$
=(ae^2-bde+cd^2)Y^2+(2ae-bd)Y+(a-d^2).$$

If $a\in P$ (hence $c,e\not\in P$, because $P\in \Gamma$), denote 

\begin{equation}\label{eq2}
F_P:=e^2g(X,\frac{-1}{e})= ae^2X^2-beX+(c-e^2).
\end{equation}


\begin{lem}\label{7}
Let $M\in Spec(A)$ and $P=M\cap B$. Then the ring $A_M$ is not regular  iff $P\supseteq (2,b)B$ and $F_PP^{-1}\subseteq M$.
\end{lem}
\begin{proof}
By Lemma \ref{6},  $A_M$ is regular if $P\not\supseteq (2,b)B$. Assume that $P\supseteq (2,b)B$. By Lemma \ref{5}, $P\in \Gamma$, so $F_P$ is defined as in (\ref{eq1}) or (\ref{eq2}).
Set $K=B/P$. We use the notations after Lemma \ref{5}.
The image of $g$ in $K[X,Y]$ is 
\begin{equation}\label{eq12}
\bar{a}X^2+\bar{c}Y^2+\bar{1}=(\bar{d}X +\bar{e}Y +\bar{1})^2
\end{equation}
%
because $\bar{b}=0$, $\bar{d}^2=\bar{a}$ and $\bar{e}^2=\bar{c}$. 
Let $Q$ be the inverse image of $M$ in $C=B[X,Y]$ and set 
\begin{equation}\label{eq3}
Z_P:={d}X +{e}Y +1.
\end{equation}
As   $P\subseteq Q$ and $g\in Q$, we get 
 $g-Z_P^2\in Q$,  so $Z_P\in Q$.  
Assume that $a\not\in P$ (the case $a\in P$, $c\not\in P$ is similar: we use (\ref{eq2}) instead of (\ref{eq1})). Then $d\not\in P$.     We have  $dX= Z_P-eY-1$, so 
$$d^2g=a(dX)^2+bd(dX)Y+d^2(cY^2-1)=$$
$$=a(Z_P-eY-1)^2+bd(Z_P-eY-1)Y+d^2(cY^2-1)=$$
$$=aZ_P^2-2aZ_P(eY+1)+bdZ_P+d^2g(\frac{-eY-1}{d},Y)=$$
\begin{equation}\label{eq11}
=aZ_P^2-2aZ_P(eY+1)+bdZ_P+F_P
\end{equation}
cf. (\ref{eq1}).
By \cite[Theorem 26, page 303]{ZS}, $A_{M}$ is not regular iff $g\in Q^2C_Q$ iff $d^2g\in Q^2C_Q$, because $d\notin Q$. 
Since $aZ_P^2$, $2aZ_P(eY+1)$ and $bdZ_P$ belong to $Q^2$, (\ref{eq11}) shows 
 that  $d^2g\in Q^2C_Q$ iff $F_P\in Q^2C_Q$.
Thus 
\begin{equation}\label{eq4}
A_{M} \mbox{ is not regular iff } F_P\in Q^2C_Q.
\end{equation}
By (\ref{eq11}), $F_P=(d^2g-aZ_P^2)+2aZ_P(eY+1)-bdZ_P$ where $d^2g-aZ_P^2\in PC$ (cf. \ref{eq12} and \ref{eq3}) and $2,b\in P$; so 
the coefficients of $F_P$ are in $P$.
%
%
Also, $P\not\subseteq Q^2C_Q$ because $C/PC=K[X,Y]$ is a regular ring. 
So, by (\ref{eq4}), $A_{M}$ is not regular iff $F_PC=P(F_PP^{-1})C\subseteq Q^2C_Q$ iff $F_PP^{-1}\subseteq QC_Q$ iff $F_PP^{-1}\subseteq Q$ iff $F_PP^{-1}\subseteq M$.
\end{proof}

\begin{thrm}\label{3}
The singular locus of $A$ is $V(H)$ where
$$H=\prod_{P\in \Gamma} (P,F_PP^{-1})A.$$
\end{thrm}
\begin{proof}
Apply Lemmas \ref{5} and \ref{7}.
\end{proof}

\begin{cor}\label{8}
The ring $A$ is regular iff   for every ${P\in \Gamma}$
we have:

$(i)$ if $a\notin P$, then $b-2de\in P^2$, $cd^2-ae^2\in P^2$, $a-d^2\notin P^2$,

$(ii)$ if $a\in P$, then $a\in P^2$, $b\in P^2$, $c-e^2\notin P^2$.
The elements $d,e$ (which depend on $P$) are defined in (\ref{eq5}).
\end{cor}
\begin{proof}
$(i)$. By Theorem \ref{3} (or Lemma \ref{7}), $A$ is regular iff  $(P,F_PP^{-1})A=A$ for every ${P\in \Gamma}$.
Fix ${P\in \Gamma}$, set $K=B/P$ and assume that  $a\notin P$.
Note that $F_P$ is defined in (\ref{eq1}); denote $F_P$ briefly by $\alpha Y^2+\beta Y+\gamma.$
%
We have the following chain of equivalences:
$$(P,F_PP^{-1})A=A  \Leftrightarrow C=(P,g,F_PP^{-1})C=(P,Z_P^2,F_PP^{-1})C \mbox{ (cf. \ref{eq3})} \Leftrightarrow $$ $$ \Leftrightarrow C=(P,Z_P,F_PP^{-1})C.$$ Hence $(P,F_PP^{-1})A=A $ iff the following ring is the zero ring 
$$C/(P,Z_P,F_PP^{-1})C\simeq B_P[X,Y]/(P,Z_P,\frac{F_P}{p})B_P[X,Y]\simeq$$ $$\simeq K[X,Y]/(Z_P,\frac{F_P}{p})K[X,Y]\simeq  K[Y]/(\frac{F_P}{p})K[Y]$$
where $p\in P$ is such that $pB_P=PB_P$; for the last isomorphism we used the fact that $d\notin P$.
Thus $(P,F_PP^{-1})A=A$ iff $\alpha/p\in pB_P$, $\beta/p\in pB_P$ and $\gamma/p\notin pB_P$ iff $\alpha \in P^2$, $\beta\in P^2$ and $\gamma\notin P^2$ iff  $ae^2+cd^2 \equiv bde$, $bd \equiv 2ae$ and $a\not\equiv d^2$ where the congruences are modulo $P^2$.
From $ae^2+cd^2 \equiv (bd)e$ and  $bd \equiv 2ae$, we get $ae^2+cd^2 \equiv bde\equiv (2ae)e$ and $bd \equiv 2d^2e$ (because $2a\equiv 2d^2$), so $cd^2 \equiv    ae^2$ and $b \equiv 2de$ because $d\notin P$.
The argument is reversible.
Case $(ii)$ can be done similarly.
\end{proof}

\begin{remark}\label{12}
Note that condition $a-d^2\notin P^2$ in Corollary  \ref{8} means that $a$ is not a quadratic residue modulo $P^2$.
\end{remark}

\begin{cor}{\em (\cite[Theorem 2]{R}.)}
Let $A=\mathbb{Z}[X,Y]/(aX^2+bXY+cY^2-1)$. Then $A$ is regular but not smooth over $\mathbb{Z}$ 
iff one of the following cases occurs (all congruences below are modulo $4$):

$(1)$ $a\equiv 3$, $b\equiv 2$, $c\equiv 3$,

$(2)$  $a\equiv 0$, $b\equiv 0$, $c\equiv 3$,

$(3)$  $a\equiv 3$, $b\equiv  0$, $c\equiv 0$.
%


\end{cor}
\begin{proof}
Assume that $A$ is  not smooth over $\mathbb{Z}$. With the notations of Corollary \ref{8}, we have $\Gamma=\{2\mathbb{Z}\}$. 
We can take $d=a$ and $e=c$. 
By Corollary \ref{8}, $A$ is regular iff

$(i)$ if $a$ is odd, then $b-2ac\in 4\mathbb{Z}$, $ac(a-c)\in 4\mathbb{Z}$, $a-a^2\notin 4\mathbb{Z}$,

$(ii)$ if $a$ is even, then $a\in 4\mathbb{Z}$, $b\in 4\mathbb{Z}$, $c-c^2\notin 4\mathbb{Z}$.
\\
The conclusion follows.
\end{proof}

\begin{example}
Consider the ring $$D=\mathbb{Z}[\theta][X,Y]/((1-\theta)X^2+\theta XY+(1-\theta)Y^2-1)$$ where  $\theta=(1+\sqrt{-7})/2$.
Using the notations above, we have: $\Gamma=\{ P\}$, where $P=\theta \mathbb{Z}[\theta]$, $a=c=1-\theta$, $b=\theta$, $d=e=1$. 
We check the conditions in part $(i)$ of Corollary \ref{8}: $b-2de=\theta-2=\theta^2\in P^2$, $cd^2-ae^2=0\in P^2$, $a-d^2=-\theta\notin P^2$. So $D$ is a regular ring. By Corollary \ref{11}, $D$ is not smooth over $\mathbb{Z}[\theta]$.
\end{example}

\begin{cor}
Let $A=\mathbb{Z}[\sqrt{-5}][X,Y]/(aX^2+bXY+cY^2-1)$ and $P=(2,1+\sqrt{-5})\mathbb{Z}[\sqrt{-5}]$.

$(i)$ $A$ is smooth over $\mathbb{Z}[\sqrt{-5}]$ iff 
$b\not\in P$ or $a,b,c\in P$.

$(ii)$ $A$ is regular but not smooth over $\mathbb{Z}[\sqrt{-5}]$ iff one of the following cases occurs:

\ \ \ \ $(1)$ $a-\sqrt{-5}$, $b$, $c-\sqrt{-5}$ are divisible by $2$,

\ \ \ \ $(2)$  $a$, $b$, $c-\sqrt{-5}$ are divisible by $2$,

\ \ \ \ $(3)$  $a-\sqrt{-5}$, $b$, $c$  are divisible by $2$.
\end{cor}
\begin{proof}
By Theorem \ref{1}, $A$ is smooth over $B=\mathbb{Z}[\sqrt{-5}]$ iff $a,c\in  \sqrt{(2, b)B}=P+bB$. Hence $(i)$ follows.
Assume that $A$ is regular but not smooth over $B$. With the notations of Corollary \ref{8}, we have $\Gamma=\{P\}$. 
Since $2B=P^2$, conditions $(i)$, $(ii)$ in Corollary \ref{8} imply that $2$ divides $b$. Assume that $a,c\notin P$. By Remark \ref{12}, 
$a\equiv c\equiv \sqrt{-5}$ (mod $2$). Assume that $a\in P$. By conditions  $(ii)$ in Corollary \ref{8}, we get 
$a\in 2B$  and $c\equiv \sqrt{-5}$ (mod $2$). Similarly, if $c\in P$, we get $c\in 2B$  and $a\equiv \sqrt{-5}$ (mod $2$).
So $(1)$,$(2)$ and $(3)$ hold. The converse follows easily from Corollary \ref{8}.
\end{proof}

The arguments in the proof of Theorem \ref{3} can be also used in certain  higher degree cases. We close our note with an example in this direction.

\begin{example}\label{14}
We show that for every prime $p$, the following ring is regular but not smooth over $\mathbb{Z}$ $$D=\mathbb{Z}[X,Y]/((p+1)X^p+p^2Y^p-1).$$
Using Euler's formula as in the proof of Theorem \ref{1}, we can prove easily that the jacobian ideal of $D$ is $J=(p,X^p-1)D\neq D$, so 
$D$ is not smooth over $\mathbb{Z}$.
Set $g=(p+1)X^p+p^2Y^p-1$ and $Z=X-1$.
Assume there exists a prime ideal $Q$ of $E=\mathbb{Z}[X,Y]$ containing $g$ such that $D_{QD}$ is not regular. As $J=(p,X^p-1)D$, $p\in Q$. Since $g$ and $Z^p$ are congruent modulo $pE$, we get $Z^p\in (g,pE)\subseteq Q$, hence $Z\in Q$. Now $X=Z+1$, so
$$g=(p+1)(Z+1)^p+p^2Y^p-1 =(p+1)[Z^p+
\sum_{k=1}^{p-1} {{p}\choose{k}}Z^{p-k}]   +p^2Y^p+p.
$$
Note that the bracket above belongs to $Q^2$ because $p,Z\in Q$. Repeating the argument in the final part of the proof of Lemma \ref{7}, we get $pY^p+1=(p^2Y^p+p)/p\in Q$, a contradiction because $p\in Q$.
\end{example}

\end{document}